\documentclass[11pt]{amsart}
\usepackage{amssymb}
\usepackage{color}
\usepackage{enumerate}
\makeatletter
\newcommand{\leqnomode}{\tagsleft@true}
\newcommand{\reqnomode}{\tagsleft@false}
\makeatother

\newtheorem{Theorem}{Theorem}[section]

\newtheorem{Problem}{Problem}
\newtheorem{Lemma}[Theorem]{Lemma}
\newtheorem{Corollary}[Theorem]{Corollary}

\newtheorem{Definition}[Theorem]{Definition}
\newcommand\supp{\mathop{\rm supp}}

\theoremstyle{definition}
\newtheorem{Remark}[Theorem]{Remark}

\begin{document}
	
	\title{Daugavet and diameter two properties in Orlicz-Lorentz spaces}
	\thanks{
		The second author was supported by Basic Science Research Program through the National Research Foundation of Korea(NRF) funded by the Ministry of Education, Science and Technology [NRF-2020R1A2C1A01010377].
		}
		\thanks{			
		The third author was supported by Basic Science Research Program through the National Research Foundation of Korea(NRF) funded by the Ministry of Education, Science and Technology [NRF-2020R1A2C1A01010377].}
	
	\author{Anna Kami\'{n}ska}
	\address{Department of Mathematical Sciences,
		The University of Memphis, TN 38152-3240}
	\email{kaminska@memphis.edu}
	
	\author{Han Ju Lee}
	\address{Department of Mathematics Education, Dongguk University - Seoul, 04620 (Seoul), Republic of Korea}
	\email{hanjulee@dgu.ac.kr}
	
	\author{Hyung-Joon Tag}
	\address{Department of Mathematics Education, Dongguk University - Seoul, 04620 (Seoul), Republic of Korea}
	\email{htag@dongguk.edu}
	\date{\today}
	
	\maketitle

	\begin{center} Dedicated to the  Organizers of the Conference "Current  Trends in Abstract and Applied Analysis", 
		May $12-15$, 2022, Ivano-Frankivst, Ukraine
	\end{center}

\begin{abstract}
	 In this article, we study the diameter two properties (D2Ps), the diametral diameter two properties (diametral D2Ps), and the Daugavet property in Orlicz-Lorentz spaces equipped with the Luxemburg norm. First, we characterize the Radon-Nikod\'ym property of Orlicz-Lorentz spaces in full generality by considering all finite real-valued Orlicz functions. To show this, the fundamental functions of their K\"othe dual spaces defined by extended real-valued Orlicz functions are computed. We also show that if an Orlicz function does not satisfy the appropriate $\Delta_2$-condition, the Orlicz-Lorentz space and its order-continuous subspace have the strong diameter two property. Consequently, given that an Orlicz function is an N-function at infinity, the same condition characterizes the diameter two properties of Orlicz-Lorentz spaces as well as the octahedralities of their K\"othe dual spaces. The Orlicz-Lorentz function spaces with the Daugavet property and the diametral D2Ps are isometrically isomorphic to $L_1$ when the weight function is regular. In the process, we observe that every locally uniformly nonsquare point is not a $\Delta$-point. This fact provides another class of real Banach spaces without $\Delta$-points. As another application, it is shown that for Orlicz-Lorentz spaces equipped with the Luxemburg norm defined by an N-function at infinity, their K\"othe dual spaces do not have the local diameter two property, and so as other (diametral) diameter two properties and the Daugavet property.
\end{abstract}

\section{Introduction}

In this article, we study the Daugavet property and diameter two properties of Orlicz-Lorentz spaces equipped with the Luxemburg norm. Orlicz-Lorentz spaces \cite{K}, are natural generalization of both Orlicz and Lorentz spaces. Various properties on Orlicz-Lorentz spaces such as rotundity \cite{K, K2}, nonsquareness \cite{FHK, FK2}, the Banach-Saks property \cite{KL}, and the Kadec-Klee property \cite{Kol} have been examined by many researchers.

Let us mention that the Orlicz-Lorentz spaces satisfying the aforementioned properties are associated with the Orlicz functions having the growth condition called $\Delta_2$. In addition, it is well-known that the $\Delta_2$-condition characterizes separable Orlicz-Lorentz spaces \cite{K} as well as the Radon-Nikod\'ym property if we restrict ourselves to a smaller class of Orlicz functions called N-functions \cite{KT}.  

From this perspective, nonseparable Orlicz-Lorentz spaces do not seem to possess many geometrical properties. However, it is recently shown that if an N-function does not satisfy the $\Delta_2$-condition then the Orlicz-Lorentz space has the diameter two property \cite{KT}. Let $X$ be a {\it real} Banach space. For $x^* \in S_{X^*}$ and $\epsilon > 0$, a {\it slice} $S(x^*, \epsilon)$ of the unit ball $B_X$ is defined by $S(x^*, \epsilon) = \{x \in B_X : x^*x > 1 - \epsilon\}$.

\begin{Definition}
	\begin{enumerate}[\rm(i)]
		\item A Banach space $X$ has the local diameter two property (LD2P) if every slice of the unit ball $B_X$ has diameter two.
		\item A Banach space $X$ has the diameter two property (D2P) if every nonempty weakly open subset of the unit ball $B_X$ has diameter two.
		\item A Banach space $X$ has the strong diameter two property (SD2P) if every convex combination of slices of the unit ball $B_X$ has diameter two.
	\end{enumerate}
\end{Definition}

It is well-known that the SD2P implies the D2P because of Bourgain's lemma which states that every relatively weakly open subset of $B_X$ contains a convex combination of slices \cite[Lemma II.1]{GGMS}. Moreover, every slice is a relatively weakly open subset of the $B_X$, so the D2P implies the LD2P. Diameter two properties are also known to be connected to other geometrical properties such as almost squareness \cite{ALL}, octahedrality \cite{HLP}, and the Daugavet property of Banach spaces. The first two properties that we just mentioned play important roles to examine the D2Ps on tensor products of Banach spaces \cite{LLZ, Ru} and on free Banach lattices \cite{DMRR}.

A Banach space $X$ has the {\it Daugavet property} if every rank-one operator $T:X \rightarrow X$ satisfies 
\[
\|I + T\| = 1 + \|T\|.
\] 
This property is not only described in terms of a specific behavior of rank-one operators but also by a specific behavior of slices of the unit ball $B_X$.

\begin{Lemma}\cite[Lemma 2.2]{KSSW}
	\label{lem:Daug}
	The following statements are equivalent.
	\begin{enumerate}[{\rm(i)}]
		\item A Banach space $(X,\|\cdot\|)$ has the Daugavet property,
		\item\label{Daugii} For every slice $S = S(x^*,\epsilon)$ where $x^*\in S_{X^*}$, every $x \in S_X$ and every $\epsilon>0$, there exists $y\in S_{X}\cap S$ such that $\|x+y\|>2-\epsilon$,
		\item\label{Daugiii} For every weak$^{*}$-slice $S^* = S(x,\epsilon)$ where $x\in S_{X}$, every $x^* \in S_{X^*}$ and every $\epsilon>0$, there exists $y^*\in S_{X^*}\cap S^*$ such that $\|x^*+y^*\|>2-\epsilon$.
	\end{enumerate}		
\end{Lemma} 

\noindent It is well-known that any Banach space with the Daugavet property has the SD2P \cite{ALN}.

Nonatomic $L_1$ and $L_{\infty}$ spaces are classical examples with the Daugavet property. More generally, any rearrangement invariant Banach function lattice over a finite, nonatomic measure space with the Daugavet property is isometrically isomorphic to $L_1$ or $L_{\infty}$ \cite{AKM, KMMW}. For the infinite measure space case, a rearrangement invariant Banach function lattice with the Daugavet property is isometrically isomorphic to $L_1$ under certain conditions \cite{AKM}, but the full description has been still unknown up to our knowledge. However, there are several results for specific Banach function lattices over an infinite measure space. For instance, Musielak-Orlicz spaces with the Daugavet property are fully characterized in \cite{KK}. Also, for finite Orlicz functions, an Orlicz space equipped with the Luxemburg norm has the Daugavet property if and only if the space is isometrically isomorphic to $L_1$ \cite{KLT}. 

Recently, a variant of the D2Ps called the diametral diameter two properties (diametral D2Ps) have gained attention from many researchers \cite{AHNTT,  AHLP, ALNT, BLZ, LP}. They are known to be weaker than the Daugavet property but stronger than the D2Ps.
\begin{Definition}
	\begin{enumerate}[\rm(i)]
		\item A Banach space $X$ has the diametral local diameter two property (DLD2P) if for every slice $S$ of the unit ball, every $x \in S \cap S_X$, and every $\epsilon > 0$ there exists $y \in B_X$ such that $\|x - y\| \geq 2 - \epsilon$.
		\item A Banach space $X$ has the diametral diameter two property (DD2P) if for every nonempty weakly open subset $W$ of the unit ball, every $x \in W \cap S_X$, and every $\epsilon > 0$, there exists $y \in B_X$ such that $\|x- y\| \geq 2 - \epsilon$. 
	\end{enumerate}
\end{Definition}
\noindent We mention that the DLD2P also describes a certain behavior of rank-one projections acting on a Banach space \cite[Theorem 1.4]{IK}. There was the diametral version of the SD2P called the DSD2P, but it is now known to be equivalent to the Daugavet property \cite{Kad}. We have the following relationship between the diametral D2Ps and the Daugavet property.
\[
\text{Daugavet Property} \implies \text{DD2P} \implies \text{DLD2P}.
\]

 But in this article, we will use the equivalent definitions of the DLD2P and the Daugavet property given by {\it $\Delta$-points} and {\it Daugavet points} \cite{AHLP} instead. For $x \in S_X$ and $\epsilon >0$, let $\Delta_{\epsilon}(x) = \{y \in B_X : \|x - y\| \geq 2 - \epsilon\}$.
\begin{Definition}
	\begin{enumerate}[\rm(i)]
		\item A point $x \in S_X$ is a $\Delta$-point if $x \in \overline{conv}\Delta_{\epsilon}(x)$ for every $\epsilon > 0$. 
		\item A point $x \in S_X$ is a Daugavet point if $B_X =  \overline{conv}\Delta_{\epsilon}(x)$ for every $\epsilon > 0$.
	\end{enumerate}
\end{Definition}
\noindent A Banach space $X$ has the DLD2P if every point on the unit sphere $S_X$ is a $\Delta$-point \cite{AHLP, IK}. Recently, these points have been characterized for different Banach spaces such as Lipschitz-free spaces \cite{JR} and certain vector-valued function spaces associated with the uniform algebras \cite{LT}.  

The article consists of four parts. In section 2 we provide necessary information about Banach function spaces and Orlicz-Lorentz spaces. In section 3, we study the SD2P of Orlicz-Lorentz spaces in view of the theory of M-ideals (Theorem \ref{th: MidealOL} and Theorem \ref{th:midealOLseq}). We characterize the Radon-Nikod\'ym property for a broader class of Orlicz-Lorentz space (Theorem \ref{th:RNPsuff}), which improves the result in \cite{KT}. As a consequence, we identify a class of Orlicz-Lorentz spaces where all D2Ps are equivalent to each other (Theorem \ref{th:equivD2P}). This particular result additionally shows the equivalence between the variants of octahedralities in the K\"othe dual of Orlicz-Lorentz spaces. In section 4, we examine the relationship between locally uniformly nonsquare points and $\Delta$-points. From the fact that locally uniformly nonsquare points are not $\Delta$-points (Theorem \ref{prop:nod}), we provide an alternative approach to the fact that Banach spaces with the DLD2P are locally octahedral. On the other extreme, this fact also shows that every locally uniformly nonsquare Banach space cannot have $\Delta$-points (Corollary \ref{cor:nodelta}). Such a class of Banach spaces includes locally uniformly rotund Banach spaces and uniformly nonsquare Banach spaces defined by James \cite{J}. With this observation, we show that the Orlicz-Lorentz function space with the Daugavet property and the diametral D2Ps is isometrically isomorphic to $L_1$ under a certain assumption on the weight function (Theorem \ref{th:OLDauUN}). We also show that the K\"othe dual $\mathcal{M}_{\varphi,w}^0$ of Orlicz-Lorentz space $\Lambda_{\varphi_*,w}$ never satisfies the LD2P where $\varphi_*$ is the complementary function of an N-function at infinity (Corollary \ref{th:KothenoLD2P}). 

\section{Preliminaries}

\subsection{Banach function spaces and sequence spaces} The space of all $\mu$-measurable extended real valued functions over a measure space ($\Omega, \Sigma,\mu)$ is denoted by $L_0(\Omega)$. In this article, we always assume either $\Omega = I = [0, \gamma), \gamma \leq \infty$, $\Sigma$ the $\sigma$-algebra of Lebesgue measurable subsets in $I$  with the Lebesgue measure $\mu = m$, or $\Omega = \mathbb{N}$ with the counting measure $m$. Unless $\gamma$ is stated specifically, any statements in this article hold for both infinite and finite interval $I$. Accordingly, we denote $L_0 = L_0(I)$ for functions and $\ell_0 = L_0(\mathbb{N})$ for sequences. The notation ``a.e.'' indicates almost everywhere convergence with respect to the Lebesgue measure $m$. A Banach space $(X, \|\cdot\|_X) \subset L_0(\Omega)$ is said to be a {\it Banach lattice} if $0 \leq |f| \leq |g|$ $\mu$-a.e. for $f \in L_0$ and $g \in X$ implies $f \in X$ and $\|f\|_X \leq \|g\|_X$. A Banach lattice $X$ is called a {\it Banach function space} if the underlying measure space $\Omega = I$ and a {\it Banach sequence space} if we consider $\Omega = \mathbb{N}$. We say a Banach lattice has the {\it Fatou property} if for every sequence $(f_n) \subset X$ such that $0 \leq |f_n| \uparrow |f|$ $\mu$-a.e. and $\sup\|f_n\|_X < \infty$ we have $f \in X$ and $\|f_n\|_X \rightarrow \|f\|_X$.  

For a Banach function lattice $X$, a function $f \in X$ is said to be {\it order-continuous} if $\|f_n\|_X \downarrow 0$ for every sequence $\{f_n\}_{n=1}^{\infty}$ of measurable functions such that $f_n \leq f$ for every $n \in  \mathbb{N}$ and $f_n\downarrow 0$ $\mu$-a.e. \cite[Proposition 1.3.5]{BS}. The set of all order-continuous elements in $X$ is denoted by $X_a$. The set $X_a$ is a closed subspace of $X$ and if $|g| \leq |f|$ $\mu$-a.e. and $f \in X_a$ then $g \in X_a$ \cite[Theorem 1.3.8]{BS}. The closure of the set of all simple functions supported on sets of finite measure is denoted by $X_b$. In general, we have $X_a \subset X_b$ \cite[Theorem 1.3.11]{BS}.

The {\it K\"othe dual} $X' \subset L_0(\Omega)$ of a Banach function lattice $X$ is a collection of  functions in $L_0(\Omega)$ such that for each $g \in X'$ the associated norm of $g \in X'$ defined by
\[
\|g\|_{X'} = \sup\left\{\left|\int_\Omega gf\right|: \|f\|_X \leq 1\right\}
\]
is finite. 

For $\lambda > 0$, the {\it distribution function} $d_f(\lambda)$ of $f \in L_0(\Omega)$ is defined by $d_f(\lambda) = \mu\{t \in \Omega: |f(t)| > \lambda\}$.   We say that $f,g\in L_0(\Omega)$ are {\it equimeasurable} whenever $d_f = d_g$ on $(0,\infty)$.  For $f \in L_0$, the function $f^*(t) = \inf\{\lambda > 0: d_f(\lambda) \leq t\}$, $t\ge 0$, is the generalized inverse of $d_f(\lambda)$, and this is called the {\it decreasing rearrangement} of $f \in L_0$. For $x \in \ell_0$, the decreasing rearrangement $x^*$ is defined by $x^*(i) = \inf\{\lambda > 0: d_x(\lambda) < i\}$, $i\in\mathbb{N}$.
 The functions $f$ and $f^*$, respectively $x$ and $x^*$,  are equimeasurable, that is,  $d_f(\lambda) = d_{f^*}(\lambda)$ for all $\lambda > 0$,  and $d_{x^*}(i) = d_x(i)$  for all $i\in\mathbb{N}$. We denote the {\it Hardy-Littlewood-P\'olya submajorization} by $f \prec g$ if $\int_{0}^t f^*  \leq \int_{0}^t g^*$ for all $t > 0$. 

For a Banach function lattice $X$, if all equimeasurable functions $f,g \in X$ satisfy $\|f\|_X = \|g\|_X$, then we say $X$ is {\it rearrangement invariant} (r.i.). Orlicz-Lorentz function and sequence spaces and their K\"othe dual spaces are examples of r.i. Banach function and sequence spaces. The fundamental function of $X$ is defined by $\phi_X(t)= \|\chi_{E_t}\|_X$ where $E_t$ is a set of measure $t \in [0, \infty]$. For more details on the theory of Banach function lattices, we refer to \cite{BS, KPS, LT2}.  

\subsection{Orlicz-Lorentz spaces and its K\"othe duals} An {\it Orlicz function} $\varphi: \mathbb{R}^+ \rightarrow \mathbb{R}^{+}$ is a convex function that is not identical to the zero function with $\varphi(0) = 0$. 
The {\it complementary function} $\varphi_*$ of an Orlicz function $\varphi$ is defined by $\varphi_*(v) = \sup\{uv - \varphi(u) : u > 0\}$, $v\ge 0$. It is well-known that $\varphi_*:\mathbb{R}^+ \to [0,\infty]$ is convex and $\varphi_{**} = \varphi$.

In this article, we consider the constants $a_{\varphi}$ and $d_\varphi$ that are given by
\begin{align*}
	a_\varphi &= \sup\{u > 0 : \varphi(u) = 0\}\\
	d_\varphi &= \sup\{u > 0 : \varphi(s)  = ks \,\, \text{for some} \,\, k \geq 0\,\,\text{and}\,\,\text{for every}\,\,s \in [0, u)\}.	
\end{align*}
We call an Orlicz function $\varphi$ to be {\it degenerate} if $a_{\varphi} > 0$. It is also easy to see that $a_{\varphi} = d_{\varphi}$ in this case. If $a_\varphi = 0$ then $\varphi$ is called {\it nondegenerate}.

Let the weight function $w: I \rightarrow [0, \gamma)$ be a decreasing, positive, locally integrable function and denote $W(t) = \int_0^t w$, $t\in I$. Therefore $W(t) <\infty$ for all $t\in [0,\gamma)$. We always assume here that if $\gamma = \infty$ then $\int_0^{\infty}w = \infty$. A convex {\it modular} $\rho_{\varphi,w}(\cdot): L_0 \rightarrow [0, \infty]$ for $f \in L_0$ is defined by
\[
\rho_{\varphi,w}(f) = \int_I \varphi(f^*)w,
\]  
and
a convex modular $\alpha_{\varphi,w}(\cdot): \ell_0 \rightarrow [0, \infty]$ for $x \in \ell_0$ by
\[
\alpha_{\varphi,w}(x) = \sum_{i=1}^\infty \varphi(x^*(i))w(i).
\] 
The modular $\rho_{\varphi,w}$ is orthogonally subadditive, that is, for $f \wedge g = 0$, $\rho_{\varphi, w}(f + g) \leq \rho_{\varphi,w}(f) + \rho_{\varphi,w}(g)$. If $0 \leq|f| \leq |g|$ a.e., then $\rho_{\varphi,w}(f) \leq \rho_{\varphi,w}(g)$. Similarly, these facts are also true for the modular $\alpha_{\varphi,w}$.

An Orlicz function $\varphi$ satisfies the {\it $\Delta_2^0$-condition} if there exist $K > 2$ and $u_0 > 0$ such that  $\varphi(2u) \leq K \varphi(u)$ for all $u \leq u_0$. Similarly, $\varphi$ satisfies the {\it $\Delta_2^{\infty}$-condition} if there exist $K > 2$ and $u_0 \ge 0$ such that  $\varphi(2u) \leq K \varphi(u)$ for all $u \geq u_0$. If $\varphi$ satisfies both the $\Delta_2^0$- and $\Delta_2^{\infty}$- conditions, we say that $\varphi$ satisfies the {\it $\Delta_2$-condition}. In this article, the ``appropriate" $\Delta_2$-condition means the $\Delta_2^\infty$-condition for the spaces considered on $I$ with $\gamma < \infty$, the $\Delta_2$-condition on $I$ with  $\gamma = \infty$, and the $\Delta_2^0$-condition  for sequence spaces.

The {\it Orlicz-Lorentz function and sequence space} $\Lambda_{\varphi,w}$ and $\lambda_{\varphi,w}$ are defined by
\begin{align*}
	\Lambda_{\varphi,w} &= \{f \in L_0 : \rho_{\varphi, w}(kf) < \infty\,\,\, \text{for some} \,\,\, k > 0\},\\
	\lambda_{\varphi,w} &= \{x \in \ell_0 : \alpha_{\varphi, w}(kx) < \infty\,\,\, \text{for some} \,\,\, k > 0\}.
\end{align*}
 We mention that $(\Lambda_{\varphi,w})_a = (\Lambda_{\varphi,w})_b$ by the fact that $\lim_{t \rightarrow 0+}\phi_{\Lambda_{\varphi, w}}(t) =\lim_{t \rightarrow 0+}\frac{1}{\varphi^{-1}(1/W(t))} = 0$ and by \cite[Theorem 2.5.5]{BS}. For the sequence space, the relationship $(\lambda_{\varphi, w})_a = (\lambda_{\varphi, w})_b$ always holds in view of \cite[Theorem 2.5.4]{BS}. Notice that we have the Orlicz space $L_{\varphi}$ when $w \equiv 1$ and the Lorentz space $\Lambda_{p, w}$ for $1 \leq p < \infty$ when $\varphi(u) = u^p$, $u\ge 0$.

In this article we only consider the Orlicz-Lorentz function and sequence spaces equipped with the {\it Luxemburg norm} $\|\cdot\|_{\varphi,w}$ that are defined by
\[
\|f\|_{\varphi,w} = \inf\left\{\epsilon > 0 : \rho_{\varphi,w}\left(\frac{f}{\epsilon}\right) \leq 1\right\}\,\,\, \text{and} \,\,\, 
\|x\|_{\varphi,w} = \inf\left\{\epsilon > 0 : \alpha_{\varphi,w}\left(\frac{x}{\epsilon}\right) \leq 1\right\},
\]
for functions and sequences, respectively. It is well known that  Orlicz-Lorentz spaces have the Fatou property.

To describe the K\"othe duals of Orlicz-Lorentz function spaces, two modulars $P_{\varphi,w}$ and $Q_{\varphi,w}$ are introduced in \cite{KR2} as follows.
\begin{align*}
	P_{\varphi,w}(f) &= \inf\left\{\int_I \varphi\left(\frac{f^*}{v}\right)v: v \prec w, v > 0, v \downarrow \right\},\\
	Q_{\varphi,w}(f) &= \int_I \varphi\left(\frac{(f^*)^0}{w}\right)w.
\end{align*}
Here $f^0$ denotes the level function of $f$ in the sense of Halperin \cite{KLR, KR2}. Then the space $\mathcal{M}_{\varphi,w}$ is defined by 
\[
\mathcal{M}_{\varphi,w} = \{f \in L_0 : P_{\varphi, w}(kf) < \infty\,\,\, \text{for some} \,\,\, k > 0\}.
\]

 For the sequence case, the modular $p_{\varphi,w}$ \cite{KR1} is similarly defined by 
	\[
	p_{\varphi,w}(x) = \inf\left\{\sum_{i=1}^{\infty} \varphi\left(\frac{x^*(i)}{v(i)}\right)v(i): v \prec w, v \downarrow\right\},
	\]
and the space $\mathfrak{m}_{\varphi,w}$ accordingly by
	\[
	\mathfrak{m}_{\varphi,w} = \{x \in \ell_0 : p_{\varphi, w}(kx) < \infty\,\,\, \text{for some} \,\,\, k > 0\}.
	\]

The Luxemburg norm $\|\cdot\|_{\mathcal{M}_{\varphi,w}}$ is defined similarly to the Orlicz-Lorentz spaces by replacing the modular $\rho_{\varphi, w}$ with $P_{\varphi,w}$. The Orlicz norm (or Amemiya norm) $\|\cdot\|_{\mathcal{M}_{\varphi,w}}^0$ and $\|\cdot\|_{\mathfrak{m}_{\varphi,w}}^0$ for the function and sequence spaces is defined by
\begin{align*}
	\|f\|_{\mathcal{M}_{\varphi,w}}^0 &= \inf_{k > 0}\frac{1}{k}(1 + P_{\varphi,w}(kf))\\
	\|x\|_{\mathfrak{m}_{\varphi,w}}^0 &= \inf_{k > 0}\frac{1}{k}(1 + p_{\varphi,w}(kx))
\end{align*}

From now on we denote the space $\mathcal{M}_{\varphi,w}$ equipped with the Luxemburg norm and the Orlicz norm by $\mathcal{M}_{\varphi,w}$ and $\mathcal{M}_{\varphi,w}^0$ respectively. We adopt the similar notations for the sequence spaces $\mathfrak{m}_{\varphi,w}$. It is well-known that the Luxemburg and Orlicz norms on the space $\mathcal{M}_{\varphi,w}$ are equivalent \cite{KLR}.   

It has been also shown that the modulars $P_{\varphi, w}$ and $Q_{\varphi,w}$ induce the same Luxemburg and Orlicz norms for the space $\mathcal{M}_{\varphi,w}$ \cite[Theorem 8.9]{KR2}. However, there is a slight difference on their modular structures, namely, we have $P_{\varphi,w}(f) \leq Q_{\varphi,w}(f)$ for every $f \in L_0(I)$. These modulars give the same value when $\varphi$ is an N-function \cite{KLR}. 

As a matter of fact, the space $\mathcal{M}_{\varphi,w}^0$ is the K\"othe dual space of an Orlicz-Lorentz space $\Lambda_{\varphi_*,w}$.

\begin{Theorem}\label{th:KLR}\cite{FLM, KLR, KR2}
	Let $\varphi$ be an Orlicz function and let $w$ be a weight function. Then the K\"othe dual space of the Orlicz-Lorentz function (resp. sequence) space  $\Lambda_{\varphi, w}$ (resp. $\lambda_{\varphi, w}$) is isometric to $\mathcal{M}_{\varphi_*,w}^0$. (resp. $\mathfrak{m}_{\varphi_*,w}^0$).
\end{Theorem}

Now, we present the formula for the fundamental function of $\mathcal{M}_{\varphi,w}$ with respect to any extended real-valued Orlicz functions where their corresponding values can be $0$ or $\infty$ on $(0, \infty)$. The similar formula with respect to a finite, strictly increasing Orlicz function has been considered in \cite[Proposition 4.9]{KR}. However, for studying certain geometrical properties of an Orlicz-Lorentz space $\Lambda_{\varphi,w}$ through its K\"othe dual $\mathcal{M}_{\varphi_*,w}$, we may encounter the cases where $\varphi_*(v) = 0$ or $\varphi_*(v) = \infty$ for some $v \in (0, \infty)$. For instance, if $\varphi(u) = ku$ for some $k > 0$, then $\varphi_*(v) = 0$ on the interval $[0, k]$ and $\varphi_*(v) = \infty$ on $(k, \infty)$.
Therefore, for the space $\mathcal{M}_{\varphi,w}$ we consider $\varphi:\mathbb{R}^{+} \rightarrow [0, \infty]$ that is a convex, left-continuous function where $\varphi(0) = 0$ and $\varphi$ is not identical to the zero function or the infinity function on $\mathbb{R}^{+}$. Here as well as the constants $a_{\varphi}$ and $d_{\varphi}$, we will also consider the constant $b_{\varphi}$ given by
\[
b_{\varphi} = \sup\{u \geq 0: \varphi(u)< \infty\}.
\]	
In this case, we may have $b_{\varphi} < \infty$. The left-continuity implies that $\lim_{u \rightarrow b_{\varphi}+} \varphi(u)= \infty$.

Even though the modulars $P_{\varphi,w}$ and $Q_{\varphi,w}$ may not provide the same corresponding values in general, they are equal for the characteristic functions $\chi_{(0, t)}$ where $t\in (0, \gamma)$.

\begin{Lemma}\label{lem:charM}
	Let $\varphi:\mathbb{R}^{+} \rightarrow [0, \infty]$  be a convex, left-continuous function where $\varphi(0) = 0$, $\varphi(\infty) = \infty$, and $\varphi$ is not identically zero on $\mathbb{R}^{+}$ and let $w$ be a weight function. For every $t \in (0, \gamma)$ and $c > 0$, we have $P_{\varphi,w}(c\chi_{(0, t)}) = Q_{\varphi,w}(c\chi_{(0, t)})$.
\end{Lemma}

\begin{proof}
	Fix $t \in (0, \gamma)$ and $c > 0$. Consider a characteristic function $\chi_{(0, t)}$. Then $((\chi_{(0, t)})^*)^0 = \frac{t w}{W(t)} \chi_{(0, t)}$, \cite[Proposition 2.2]{KT}.
	Then notice that
	\[
	Q_{\varphi,w}(c\chi_{(0, t)}) = \int_I \varphi\left(\frac{c((\chi_{(0, t)})^*)^0}{ w}\right)w = \varphi\left(\frac{ct}{ W(t)}\right)W(t) = \int_I\varphi\left(\frac{c\chi_{(0, t)}}{(W(t)/t) \chi_{(0,t)}}\right)(W(t)/t) \chi_{(0,t)}.
	\]
	Since $v = \frac{W(t)}{t}\chi_{(0,t)} \prec w$ and $v$ is decreasing, we see that $P_{\varphi,w}(c\chi_{(0, t)})\leq Q_{\varphi,w}(c\chi_{(0, t)})$.
	
	On the other hand, for every $v \prec w$, $v > 0$, $v\downarrow$ with $V(t) =\int_0^t v$, by Jensen's inequality, we have 
	\[
	\int_I \varphi\left(\frac{c\chi_{(0, t)}(s)}{v(s)}\right)\frac{v(s)}{V(t)}ds \geq \varphi\left(\frac{\int_Ic\chi_{(0, t)}(s)ds}{V(t)}\right) = \varphi\left(\frac{ct}{V(t)}\right).
	\]
	Since $V(t) \leq W(t)$ and the function $s \mapsto \varphi(\frac{c}{s})s$ is decreasing on $(0,\infty)$ for $c > 0$, we have
	\[
	\int_I \varphi\left(\frac{c\chi_{(0, t)}(s)}{v(s)}\right)\frac{v(s)}{V(t)}ds \geq \varphi\left(\frac{ct}{W(t)}\right)W(t) = Q_{\varphi,w}(c\chi_{(0, t)}).
	\]
    This inequality holds for every appropriate $v$, so we have $P_{\varphi,w}(c\chi_{(0, t)}) \geq Q_{\varphi,w}(c\chi_{(0, t)})$. Therefore, the equality holds. 
\end{proof}

\begin{Theorem}\label{th:fundM}
	Let $\varphi:\mathbb{R}^{+} \rightarrow [0, \infty]$ be a convex, left-continuous function where $\varphi(0) = 0$, $\varphi(\infty) = \infty$, and $\varphi$ is not identically zero on $\mathbb{R}^{+}$ and let $w$ be a weight function. Then the fundamental function $\phi_{\mathcal{M}_{\varphi,w}}$ is expressed by
		\[
		\phi_{\mathcal{M}_{\varphi,w}}(t) = \frac{t/W(t)}{\varphi^{-1}(1/W(t))},
		\]
		where $\varphi^{-1}$ is the inverse function of $\varphi$ restricted to the interval $(a_{\varphi}, b_{\varphi}]$. For the case of $b_{\varphi} = \infty$, such inverse function is defined with respect to $\varphi$ restricted to the interval $(a_{\varphi}, \infty)$.  
\end{Theorem}
\begin{proof}
	
	Fix $t \in (0, \gamma)$ and consider the characteristic function $\chi_{(0, t)}$. Then by Lemma \ref{lem:charM}, we have 
		\[
	P_{\varphi,w}\left(\frac{\chi_{(0, t)}}{c}\right) = \int_I \varphi\left(\frac{((\chi_{(0, t)})^*)^0}{c w}\right)w = \varphi\left(\frac{t}{c W(t)}\right)W(t)
	\]
    for every $c > 0$.
     
	Now, assume that $b_{\varphi} < \infty$. We claim that $\frac{t}{cW(t)} \in (a_\varphi, b_{\varphi}]$ whenever $c > 0$ and $P_{\varphi,w}\left(\frac{\chi_{(0, t)}}{c}\right) \leq 1$. Indeed, if $\frac{t}{cW(t)} > b_{\varphi}$,
		\[
		\infty = \varphi\left(\frac{t}{cW(t)}\right) \leq \frac{1}{W(t)},
		\]
 which is a contradiction. This shows that $\frac{t}{cW(t)} \le b_{\varphi}$. 
	
	On the other hand, for every $c > 0$ that satisfies $P_{\varphi,w}\left(\frac{\chi_{(0, t)}}{c}\right) \leq 1$, if we assume that $\frac{t}{cW(t)} \leq a_{\varphi}$, then $\frac{t}{W(t)a_{\varphi}} \leq c$. Then we have that $d := \frac{t}{W(t)a_{\varphi}} \leq \|\chi_{(0, t)}\|_{\mathcal{M}_{\varphi,w}}$ by the definition of the Luxemburg norm. Hence we also have 
		\[
		P_{\varphi,w}\left(\frac{\chi_{(0, t)}}{d}\right) \geq 1.
		\]
	
	But then, we also see that
	\[
	P_{\varphi,w}\left(\frac{\chi_{(0, t)}}{d}\right) = \varphi\left(\frac{t}{W(t)}\cdot \frac{W(t)}{t} \cdot a_{\varphi}\right)W(t) = 0,
	\]
	which leads to a contradiction. Hence, this shows that $a_{\varphi} < \frac{t}{cW(t)}$.

	Since $\frac{t}{cW(t)} \in (a_{\varphi}, b_{\varphi}]$ for every $c > 0$ that satisfies $P_{\varphi,w}\left(\frac{\chi_{(0, t)}}{c}\right) \leq 1$, we have $\varphi_{|(a_{\varphi},b_{\varphi}]}\left(\frac{t}{cW(t)}\right) = \varphi\left(\frac{t}{cW(t)}\right)$. Denote $\varphi = \varphi_{|(a_{\varphi},b_{\varphi}]}$. Hence this gives us the desired formula for $\phi_{\mathcal{M}_{\varphi,w}}(t)$ for $t \in (0, \infty)$. 
	
	If $\frac{t}{cW(t)} = b_{\varphi} = \infty$, then we immediately see that \[
	\infty = \varphi(\infty) =  \varphi\left(\frac{t}{cW(t)}\right) \leq \frac{1}{W(t)},
	\]
	which also leads to a contradiction. Hence, $\frac{t}{cW(t)} < b_{\varphi}$. Therefore, by using the same argument, we show that the same formula holds with $\varphi^{-1}$ that is the inverse function of $\varphi$ restricted to $(a_{\varphi}, \infty)$.
\end{proof}

\section{Diameter two properties on Orlicz-Lorentz function and sequence spaces}

\subsection{Orlicz-Lorentz spaces with the Radon-Nikod\'ym property}

Let us recall that a Banach space $X$ with the Radon-Nikod\'ym property (RNP) has slices of the unit ball with arbitrarily small diameter. Sufficient conditions for a Banach function space to have the RNP are given in \cite{KLT} as follows.

\begin{Theorem}\cite{KLT}\label{th:RNKothe}
	Let $X$ be a Banach function space  over a complete $\sigma$-finite measure space $(\Omega, \Sigma, \mu)$.
	\begin{itemize}
		\item[(i)]
		If $X$ has the RNP then $X$ is order-continuous.
		\item[(ii)]
		Assume that $X$ has the Fatou property and $(X')_a = (X')_b$.  Then if $X$ is order-continuous then $X$ has the RNP.
	\end{itemize}
\end{Theorem}

We mention that Theorem \ref{th:RNKothe} was implicitly used to characterize the Orlicz-Lorentz spaces $\Lambda_{\varphi,w}$ defined by an N-function $\varphi$ satisfying the RNP in relation to the appropriate $\Delta_2$-condition \cite{KT}. Here we consider all finite Orlicz functions and provide a characterization of the Orlicz-Lorentz spaces with the RNP.

We say an Orlicz function is an {\it N-function at infinity} if $\lim_{u \rightarrow \infty} \frac{\varphi(u)}{u}= \infty$. When $\varphi$ is not an N-function at infinity, we observe that the Orlicz-Lorentz space fails to have the RNP under a certain assumption on the weight function. To show this, we will use another formula for the norm in $\Lambda_{1, w}$ \cite[Section II.5, pg 111]{KPS}, that is,
\[
\|f\|_{1,w} = \int_0^{\infty}W(d_f(\lambda))d\lambda. 
\]	
Recall that $d_f(\lambda)$, $\lambda\in [0,\infty)$, is the distribution function of $f$.

\begin{Lemma}\label{lem:OLnoRNP}
	Let $\varphi$ be an Orlicz function and let $w$ be a weight function on $I =  [0, \gamma)$, where $\gamma < \infty$. If $\varphi$ is not an N-function at infinity and $\lim_{t \rightarrow 0+} \frac{W(t)}{t} = c > 0$, then the Orlicz-Lorentz space $\Lambda_{\varphi, w}(0,\gamma)$ coincides with $L_1(0, \gamma)$ as sets with equivalent norms.
\end{Lemma}

\begin{proof}
		From the fact that $\frac{W(t)}{t}$ is decreasing and $\lim_{t \rightarrow 0+}\frac{W(t)}{t} = c > 0$, we have $\frac{W(\gamma)}{\gamma} \leq \frac{W(t)}{t}\leq c$ for every $t \in [0, \gamma]$. Thus
		\begin{equation}\label{eq:3.7.1}
		\frac{W(\gamma)}{\gamma} \cdot t \leq W(t)\leq c\cdot t
		\end{equation}
		for every $t \in [0, \gamma]$. 
		Since $\varphi$ is not an N-function at infinity and the function $\varphi(u)/u$ is increasing, we have $\lim_{u\to\infty} \varphi(u)/u = K < \infty$, so $\varphi(u) \leq Ku$ for every $u \geq 0$. Moreover, for a given $0 < M < K$, there exists $u_0 \ge 0$ such that $\varphi(u) \geq Mu$ for every $u \geq u_0$.
	
		Let $f \in \Lambda_{\varphi,w}(0, \gamma)$ such that $\|f\|_{\varphi,w} = 1$ and let $E = \{t \in I: |f(t)| \geq u_0\}$. Since $f^*$ and $f$ are equimeasurable, the set $\{t \in I: f^*(t) \geq u_0\}$ is actually an interval $[0, mE)$. Hence, we see that $
		f^*(t) \le \frac{1}{M} \varphi(f^*(t)) \,\,\,\text{for} \,\,\, t < mE$ and
		$f^*(t) < u_0 \,\,\, \text{for} \,\,\, t \geq mE  $.
		Then we obtain
		\begin{eqnarray*}
			\frac{W(\gamma)}{\gamma}\|f\|_1  = \frac{W(\gamma)}{\gamma} \cdot \int_0^{\infty}d_{f^*}(s)ds &\leq& \int_0^{\infty} W(d_{f^*}(s))ds = 	\|f\|_{1,w}\\
			&=& \int_0^{mE} f^*(s)w(s)ds + \int_{mE}^{\gamma} f^*(s)w(s)ds\\
			&\leq&  \int_0^{mE} f^*(s)w(s)ds + u_0 \cdot W(\gamma)\\
			&\leq& \frac{1}{M}\int_0^{mE}\varphi(f^*(s))w(s)ds + u_0 \cdot W(\gamma)\\
			&\leq& \frac{1}{M}\rho_{\varphi, w}(f) + u_0 \cdot W(\gamma)\\
			&\leq& \frac{1}{M} + u_0 \cdot W(\gamma) = C < \infty. 
		\end{eqnarray*}
		Hence, $\frac{W(\gamma)}{C\gamma}\cdot \|f\|_1 \leq \|f\|_{\varphi,w}$. 
	
		Now, since $\varphi(u) \leq Ku$ for every $u > 0$, for a given subset $F \subset [0, \gamma)$ we have
		\[
		\int_0^{\gamma} \varphi\left(\left(\frac{\chi_F}{KW(mF)}\right)^*\right)w  = \int_0^{mF} \varphi\left(\frac{1}{KW(mF)}\right)w = \varphi\left(\frac{1}{KW(mF)}\right) W(mF) \leq 1,
		\]
		and so $\|\chi_F\|_{\varphi,w} \leq K\cdot W(mF)$. Moreover, by (\ref{eq:3.7.1})  we have $\|\chi_F\|_{1,w} = W(mF) \leq c\cdot mF = c \cdot \|\chi_F\|_1$, and so $ \|\chi_F\|_{\varphi,w} \leq c \cdot K\|\chi_F\|_1$. 
		Now, consider a simple function $f = \sum_{i=1}^n a_i \chi_{F_i}$ where $a_1 > a_2 > \cdots > a_n > 0$ and $F_i$'s are pairwise disjoint measurable subsets of $(0, \gamma)$. Then 
		\[
		\|f\|_{\varphi, w} \leq \sum_{i=1}^{n} a_i\|\chi_{F_i}\|_{\varphi,w} \leq c \cdot K\cdot \sum_{i=1}^{n} a_i mF_i = c\cdot K\cdot \|f\|_1.
		\]
		Now, consider a function $f \in L_1(0, \gamma)$. Then there exists an increasing sequence $(f_n)$ of non-negative simple functions with support of finite measure such that $f_n \uparrow |f|$ a.e. Then  by the Fatou property of $L_1(0, \gamma)$ and $\Lambda_{\varphi,w}(0, \gamma)$ we have  that $\|f\|_{\varphi,w} \leq c \cdot K \cdot \|f\|_1$ for every $f \in L_1(0, \gamma)$. 
		Consequently, $\Lambda_{\varphi, w}(0, \gamma)$ and $L_1(0, \gamma)$ coincide as sets with equivalent norms.
\end{proof}

\begin{Theorem}\label{cor:OLnoRNP}
	Let $\varphi$ be an Orlicz function and let $w$ be a weight function. If $\varphi$ is not an N-function at infinity and $\lim_{t \rightarrow 0+} \frac{t}{W(t)} > 0$, then the Orlicz-Lorentz space $\Lambda_{\varphi,w}$ does not have the RNP.
\end{Theorem}

\begin{proof}
		If $\gamma < \infty$, we immediately obtain the result in view of Lemma \ref{lem:OLnoRNP}. Now, suppose that $\gamma = \infty$ and assume to the contrary that $\Lambda_{\varphi,w}(0, \gamma)$ has the RNP. For a given $0 < \alpha < \gamma$, consider an isometric embedding $T: \Lambda_{\varphi,w}(0, \alpha) \rightarrow \Lambda_{\varphi,w}(0, \gamma)$ defined by $Tf = f\chi_{(0, \alpha)}$. Then $\Lambda_{\varphi,w}(0, \alpha)$ is a closed subspace of $\Lambda_{\varphi,w}(0, \gamma)$. In view of Lemma \ref{lem:OLnoRNP}, the space $\Lambda_{\varphi,w}(0, \alpha)$ coincides with $L_1(0, \alpha)$ as sets and have equivalent norms. By using the fact that the RNP is inherited by closed subspaces and that the property is invariant under isomorphisms, $L_1(0, \alpha)$ would need to have the RNP, which is a contradiction.	
\end{proof}

\begin{Remark}A non-negative function $f$ is said to be {\it equivalent to a linear function near infinity (resp. zero)} if there exist $a, b > 0$, and $t_0 > 0$ such that $at \leq f(t) \leq bt$ for every $t \geq t_0$ (resp. $0 \leq t \leq t_0$).  Now clearly we can rephrase Theorem \ref{cor:OLnoRNP} as follows.
{\it If $\varphi$ is equivalent to a linear function near infinity and $W$ is equivalent to a linear function near zero, then the Orlicz-Lorentz space $\Lambda_{\varphi, w}$ does not have the RNP.}
\end{Remark}

\begin{Remark}
Let us mention that an Orlicz function $\varphi$ not being an N-function at infinity does not fail the RNP entirely. For instance, if $\lim_{t \rightarrow 0+}\frac{t}{W(t)} = 0$, then the Lorentz space $\Lambda_{1,w}$ has the RNP \cite[Proposition 4.1]{AKM2}.
\end{Remark}

\begin{Theorem}\label{th:RNPsuff}
	Let $\varphi$ be an Orlicz function and let $w$ be a weight function. Then the Orlicz-Lorentz space has the RNP if and only if the following two conditions are satisfied:
		\begin{enumerate}[\rm(i)]
			\item Either $\varphi$ is an N-function at infinity or $\lim_{t \rightarrow 0+}\frac{t}{W(t)} = 0$.
			\item $\varphi$ satisfies the appropriate $\Delta_2$-condition.  
		\end{enumerate}
\end{Theorem}

\begin{proof}
	From Theorem \ref{cor:OLnoRNP}, we see that if $\Lambda_{\varphi,w}$ has the RNP then either $\varphi$ is an N-function at infinity or $\lim_{t \rightarrow 0+}\frac{t}{W(t)} = 0$. In addition, $\Lambda_{\varphi,w}$ is order-continuous in view of Theorem \ref{th:RNKothe}.(i). Hence $\varphi$ satisfies the appropriate $\Delta_2$-condition.
	
	Conversely, we first show that either conditions in (i) implies $\lim_{t \rightarrow 0+}\phi_{\mathcal{M}_{\varphi_*,w}}(t) = 0$. Assume that the function $\varphi$ is an N-function at infinity. Then its conjugate function $\varphi_*$ is finite \cite[Lemma 3.4.(a)]{KLT}, and so $b_{\varphi_*} = \infty$. Then by Theorem \ref{th:fundM}, we have 
		\[
	\phi_{\mathcal{M}_{\varphi_*,w}}(t) = \|\chi_{(0, t)}\|_{\mathcal{M}_{\varphi_*,w}} =  \frac{t/W(t)}{\varphi_*^{-1}(1/W(t))},
	\]
	where $\varphi_*^{-1}$ is the inverse function of $\varphi_*$ restricted to the interval $(a_{\varphi_*}, \infty)$.	
	Since the function $t \mapsto \frac{t}{W(t)}$ is increasing, we have $\lim_{t \rightarrow 0+}\frac{t}{W(t)} = L \in [0, \infty)$. Moreover, $\varphi_*^{-1}(1/W(t)) \rightarrow b_{\varphi_*} =\infty$ as $t \rightarrow 0+$. This implies that $\lim_{t \rightarrow 0+}\phi_{\mathcal{M}_{\varphi_*,w}}(t) = 0$.
	
	Now, assume that $\varphi$ is not an N-function at infinity. Then there exists $0 < K < \infty$ such that $\varphi(u) \leq Ku$. Hence, we have $\varphi_*(v) = \sup_{u > 0}\{uv - \varphi(u)\}\geq \sup_{u > 0}\{(v - K)u\}$. This shows that $\varphi_*(v) = \infty$ for every $v > K$, and so $b_{\varphi_*} = K$. Since $\frac{1}{W(t)} \rightarrow \infty$ as $t \rightarrow 0+$, we have $\varphi_*^{-1}(\frac{1}{W(t)}) \rightarrow b_{\varphi_*}$ as $t \rightarrow 0+$. Thus if $\lim_{t \rightarrow 0+}\frac{t}{W(t)} = 0$, we obtain 
		\[
		\lim_{t \rightarrow 0+}\phi_{\mathcal{M}_{\varphi_*,w}}(t) =  \frac{\lim_{t \rightarrow 0+}(t/W(t))}{\lim_{t \rightarrow 0+}\varphi_*^{-1}(1/W(t))} = \frac{0}{b_{\varphi_*}} = 0.
		\]
	
	We have shown that $\lim_{t \rightarrow 0+}\phi_{\mathcal{M}_{\varphi_*,w}}(t) = 0$ when either conditions in (i) hold. Since the Luxemburg norm and the Orlicz norm for the space $\mathcal{M}_{\varphi_*, w}$ are equivalent, we also get that $\lim_{t \rightarrow 0+}\phi_{\mathcal{M}_{\varphi_*,w}^0}(t) = 0$. Hence $((\Lambda_{\varphi,w})')_a = (\mathcal{M}_{\varphi_*,w}^0)_a = (\mathcal{M}_{\varphi_*,w}^0)_b = ((\Lambda_{\varphi,w})')_b$ by Theorem \ref{th:KLR}. Now we assume additionally that $\varphi$ satisfies the appropriate $\Delta_2$-condition. Then $\Lambda_{\varphi,w}$ is order-continuous \cite[Theorem 2.4]{K}. Therefore, the space $\Lambda_{\varphi,w}$ has the RNP by Theorem \ref{th:RNKothe}.(ii).
\end{proof}

We also have the sequence analogue of Theorem \ref{th:RNPsuff}.
\begin{Theorem}
	Let $\varphi$ be an Orlicz function and $w$ be a weight function. Then the Orlicz-Lorentz sequence space $\lambda_{\varphi, w}$ has the RNP if and only if $\varphi$ has the $\Delta_2^0$-condition.
\end{Theorem}

\begin{proof}
	From the fact that the space $\mathfrak{m}_{\varphi_*,w}^0$  is rearrangement invariant Banach sequence space and has the Fatou property, we always have $(\mathfrak{m}_{\varphi_*,w}^0)_a = ((\lambda_{\varphi,w})')_a = ((\lambda_{\varphi,w})')_b = (\mathfrak{m}_{\varphi_*,w}^0)_b$ by \cite[Theorem 2.5.4]{BS}. If $\varphi$ satisfies the $\Delta_2^0$-condition, then $\lambda_{\varphi, w}$ is order-continuous \cite{K}. Hence by Theorem \ref{th:RNKothe}.(ii), the space $\lambda_{\varphi, w}$ has the RNP.
	
	Conversely, if $\lambda_{\varphi,w}$ has the RNP, then the space is order-continuous by Theorem \ref{th:RNKothe}.(i). Since $\lambda_{\varphi,w}$ is order-continuous if and only if $\varphi$ satisfies the $\Delta_2^0$-condition \cite{K}, we obtain the desired claim. 
\end{proof}

The following corollary is an immediate consequence of Theorem \ref {th:RNPsuff} to $\varphi(u) = ku$, with $k > 0$.

\begin{Theorem}\label{th:lorRNPchar}
	Let $w$ be a weight function. The Lorentz space $\Lambda_{1,w}$ has the RNP if and only if $\lim_{t \rightarrow 0+}\frac{t}{W(t)} = 0$.
\end{Theorem}

\subsection{The diameter two properties of Orlicz-Lorentz spaces}
Now, we study the diameter two properties of Orlicz-Lorentz spaces $\Lambda_{\varphi, w}$ and $\lambda_{\varphi,w}$ defined by N-functions at infinity. Here the theory of M-ideals plays an important role. A closed subspace $Y$ is said to be an {\it M-ideal} in $X$ if $Y^{\perp}$ is the range of the bounded projection $P:X^* \rightarrow X^*$ such that $\|F\|_{X^*} = \|P(F)\|_{X^*} + \|(I - P)(F)\|_{X^*}$ for $F \in X^*$.

If a Banach function lattice $X$ has the Fatou property and $X_a = X_b$, then the K\"othe dual $X'$ is isometrically isomorphic to $(X_a)^*$ \cite[Corollary 1.4.2]{BS}. Moreover, there exists a unique decomposition of a bounded linear functional $F = H + S$, where $H$ and $S$ are regular and singular linear functionals, respectively \cite{Z}. Since any singular functional returns zero for every order-continuous element, we see that $X^* = (X_a)^* \oplus (X_a)^{\perp}$ is isometrically isomorphic to $X'\oplus X_{s}^*$, where $X_{s}^*$ denotes the set of all singular functionals on $X$. Hence, we can say that the order-continuous subspace $X_a$ is an M-ideal in $X$ if $\|F\|_{X^*} = \|H\|_{X^*} + \|S\|_{X^*} = \|h\|_{X'} + \|S\|_{X^*}$.

For the class of Orlicz-Lorentz spaces defined by N-functions, it has been shown that the order-continuous subspaces $(\Lambda_{\varphi,w})_a$ and $(\lambda_{\varphi,w})_a$ are M-ideals in $\Lambda_{\varphi,w}$ and $\lambda_{\varphi,w}$, respectively \cite{KLT2}. Here we extend these results to a broader class of Orlicz-Lorentz function and sequence spaces by using the relationship between M-ideals and intersection of balls. Our results generalize the analogous results on Orlicz function and sequence spaces in \cite{HWW}.  

\begin{Lemma} \cite[Theorem I.2.2]{HWW}\label{Werner}
	For a Banach space $X$, the following statements are equivalent.
	\begin{enumerate}[\rm(i)]
		\item A subspace $Y$ is an M-ideal in $X$.
		\item (3-ball property) For all $y_1, y_2, y_3 \in B_Y$, for all $x \in B_X$ and $\epsilon >0$ there exists $y \in Y$ such that $\|x + y_i - y\| \leq 1 + \epsilon$ for $i = 1,2,3$. 
	\end{enumerate}
\end{Lemma}

\begin{Theorem}\label{th: MidealOL}
	Let $\varphi$ be a nondegenerate Orlicz function and let $w$ be a weight function. Then the order-continuous subspace $(\Lambda_{\varphi,w})_a$ is an M-ideal in $\Lambda_{\varphi,w}$.
\end{Theorem}
\begin{proof}
		Let $f \in B_{\Lambda_{\varphi,w}}$. First, notice that $d_f(\lambda) < \infty$ for every $\lambda > 0$. Indeed, suppose that there exists $\lambda > 0$ such that $d_{f}(\lambda) = m\{|f(t)| > \lambda\} = \infty$. By the equimeasurability between $f$ and $f^*$, we have $m\{f^*(t) > \lambda\} = \infty$. From the fact that $f \in B_{\Lambda_{\varphi,w}}$, there exists $k > 0$ such that $\rho_{\varphi, w}(kf) < \infty$. But then we have
		\[
		\int_I \varphi(kf^*)w \geq \int_0^{m\{f^*(t) > \lambda\}} \varphi(k\lambda)w = \varphi(k\lambda)\cdot \int_0^{\infty}w = \infty,
		\]  
		which is a contradiction.
		
		Now, define $f_n = f \chi_{\{\frac{1}{n} \leq |f| \leq n\}}$, $n\in\mathbb{N}$. Since $f - f_n \downarrow 0$ a.e., $|f_n| \leq |f|$ a.e., and $d_f(\lambda) < \infty$ for every $\lambda > 0$, we have $(f - f_n)^* \rightarrow 0$ a.e. \cite[Property 12$^{\circ}$, pg 67]{KPS}. So we obtain $\rho_{\varphi,w}(f- f_n) \rightarrow 0$ by the Lebesgue dominated convergence theorem. Hence for any $\delta > 0$, there exists $N_1 \in \mathbb{N}$ such that $\rho_{\varphi,w}(f - f_{N_1}) < \delta$. 
		
		Now, let $g_i \in B_{(\Lambda_{\varphi,w})_a}, i = 1,2,3$, and $\epsilon > 0$ be arbitrary. From the fact that $\chi_{\{|f| > n\} \cup \{|f| < \frac{1}{n}\}} \rightarrow 0$ a.e., we have $g_i - g_i \chi_{\{\frac{1}{n} \leq |f| \leq n\}} = g_i \chi_{\{|f| < \frac{1}{n}\} \cup \{|f| > n\}}\rightarrow 0$ a.e. by the order-continuity of $g_i$. Hence, for every $\epsilon > 0$, we can choose $N_2 \in \mathbb{N}$ such that $\|g_i \chi_{\{|f| < \frac{1}{N_2}\}\cup\{|f| > N_2\}}\|_{\varphi,w} \leq \frac{\epsilon^2}{1 + \epsilon}$. By letting $N = \max\{N_1, N_2\}$, we have
		\begin{equation}\label{eq:rhocontrol}
			\rho_{\varphi,w}(f - f_{N}) < \delta\,\,\, \text{and} \,\,\, \|g_i \chi_{\{|f| < \frac{1}{N}\}\cup\{|f| > N\}}\|_{\varphi,w} \leq \frac{\epsilon^2}{1 + \epsilon}.
		\end{equation}
		Moreover, by the convexity of $\varphi$, we see that
		\[
		\rho_{\varphi, w}\left(\frac{g_i \chi_{\{|f| < \frac{1}{N}\}\cup\{|f| > N\}}}{\epsilon}\right) \leq \frac{\epsilon}{1 + \epsilon}.
		\]
		
		Now, let $\tilde{g_i} = \frac{1}{1 + \epsilon}g_i \chi_{\{\frac{1}{N}\leq |f| \leq N\}}$. Since $\rho_{\varphi,w}(g_i) \leq 1$, we have
		\begin{eqnarray*}
			\rho_{\varphi, w}\left(\frac{g_i - \tilde{g}_i}{\epsilon}\right) &=& \rho_{\varphi,w} \left(\frac{g_i\chi_{\{|f| < \frac{1}{N}\}\cup\{|f| > N\}} + \frac{\epsilon}{1 + \epsilon}g_i\chi_{\{\frac{1}{N}\leq |f| \leq N\}}}{\epsilon}\right)\\
			&\leq& \rho_{\varphi, w}\left(\frac{g_i\chi_{\{|f| < \frac{1}{N}\}\cup\{|f| > N\}}}{\epsilon}\right) + \frac{1}{1 + \epsilon} \rho_{\varphi,w}(g_i) \leq 1.
		\end{eqnarray*}
		This shows that $\|g_i - \tilde{g}_i\|_{\varphi,w} \leq \epsilon$. Moreover, since $\rho_{\varphi, w}(\tilde{g}_i) \leq \frac{1}{1 + \epsilon} < 1$, we choose $\delta > 0$ such that $\rho_{\varphi, w}(\tilde{g}_i) = 1 - \delta$.
		
		To finish the proof, let $h = f_N$. Notice that $\supp \tilde{g}_i \subset \supp h$. By the orthogonal subadditivity of $\rho_{\varphi, w}$ and  (\ref{eq:rhocontrol}), we have
		\begin{eqnarray*}
			\rho_{\varphi, w}(f + \tilde{g}_i - h) = \rho_{\varphi,w}(f- f_N + \tilde{g}_i) &\leq& \rho_{\varphi, w}(f - f_N) + \rho_{\varphi, w}(\tilde{g}_i)\\
			&\leq& \delta + 1 - \delta = 1.
		\end{eqnarray*}
		Hence we have $\|f + \tilde{g}_i - h\|_{\varphi,w} \leq 1$, and so
		\[
		\|f + g_i - h\|_{\varphi,w} \leq \|f + \tilde{g}_i - h\|_{\varphi,w} + \|g_i - \tilde{g}_i\|_{\varphi,w} \leq 1 + \epsilon. 
		\]  
		Therefore, by Lemma \ref{Werner}, we obtain the desired result.
	\end{proof}

We also have the sequence analogue of the Theorem \ref{th: MidealOL}.

\begin{Theorem}\label{th:midealOLseq}
	Let $\varphi$ be a nondegenerate Orlicz function and let $w$ be a weight sequence. Then the order-continuous subspace $(\lambda_{\varphi,w})_a$ is an M-ideal in $\lambda_{\varphi,w}$.
\end{Theorem}

	\begin{proof}	
		Let $a =(a(i)) \in B_{\Lambda_{\varphi,w}}$, $x_j = (x_j(i)) \in B_{(\Lambda_{\varphi,w})_a}$, where $j = 1,2,3$, and $\epsilon >0$ be arbitrary. From the fact that $(\lambda_{\varphi, w})_a = (\lambda_{\varphi, w})_b$, without loss of generality, assume that there exists $N_1 \in \mathbb{N}$ such that $x_j(i) = 0$ for $i > N_1$ for some $N_1 \in \mathbb{N}$. Now, we define $\bar{x_j} = (\bar{x_j}(i))$ where $\bar{x_j}(i) = \frac{x_j(i)}{1 + \epsilon}$, $i \in \mathbb{N}$. Then by the convexity of $\varphi$, we have
		\[
		\alpha_{\varphi,w}\left(\frac{x_j-\bar{x_j}}{\epsilon}\right) = \alpha_{\varphi,w}\left(\frac{x_j}{1 + \epsilon}\right) \leq  \frac{1}{1 + \epsilon}\alpha_{\varphi,w}(x_j) < 1,
		\]
		and so 
		\[
		\|x_j - \bar{x_j}\|_{\varphi,w} \leq \epsilon.
		\] 
		Moreover, we have 
		\[
		\alpha_{\varphi,w}(\bar{x_j}) = \alpha_{\varphi,w}\left(\frac{x_j}{1 + \epsilon}\right) \leq \frac{1}{1 + \epsilon}\alpha_{\varphi,w}(x_j) < 1.
		\]
		Hence there exists $\delta > 0$ such that $\alpha_{\varphi,w}(\bar{x_j}) = 1 - \delta$.
		
		Now, let $a_k = a \chi_{\{1,2,\dots,k\}}$. Since $d_a(\lambda) < \infty$ for every $\lambda > 0$, we have $(a - a_k)^* \rightarrow 0$ because $|a - a_k| \rightarrow 0$ as $k\to\infty$. Since $\|a\|_{\varphi,w} \le 1$, $\alpha_{\varphi,w}(a) \le 1$.  Hence for the $\delta > 0$, we can choose $N_2 > N_1$ such that 
		
		\begin{equation}\label{eq:alphacontrol}
			\alpha_{\varphi,w}(a\chi_{\{N_2, N_2 + 1, \dots\}}) = \sum_{i=1}^{\infty} \varphi((a\chi_{\{N_2, N_2 + 1, \dots\}})^*(i))w(i) < \delta
		\end{equation} 
		by the Lebesgue dominated convergence theorem. Define $y = (y(i))_{i=1}^{\infty}$ where $y(i) = a(i)$ for $i < N_2$ and $y(i) = 0$ otherwise. 
		
		By the orthogonal subadditivity of $\alpha_{\varphi,w}$ and (\ref{eq:alphacontrol}), we obtain
		\begin{eqnarray*}
			\alpha_{\varphi,w}(a + \bar{x_j} - y) &=& \alpha_{\varphi,w}\left(a\chi_{\{N_2, N_2 + 1, \dots\}} + \frac{x_j}{1 + \epsilon}\chi_{\{1,2,\dots,N_1\}}\right)\\
			&\leq& \alpha_{\varphi,w}(a\chi_{\{N_2, N_2 + 1, \dots\}}) + \alpha_{\varphi, w}\left(\frac{x_j}{1 + \epsilon}\chi_{\{1,2,\dots,N_1\}}\right)\\
			&=& \alpha_{\varphi,w}(a\chi_{\{N_2, N_2 + 1, \dots\}}) + \alpha_{\varphi,w}(\bar{x_j})\\ 
			&<& 1 - \delta + \delta = 1. 
		\end{eqnarray*}

		Hence $\|a + \bar{x_j} - y\|_{\varphi,w} \leq 1$. From the fact that $\|x_j - \bar{x_j}\|_{\varphi,w} \leq \epsilon$, we get $\|a  + x_j - y\|_{\varphi,w} \leq \|a + \bar{x_j} - y\|_{\varphi,w} + \|x_j - \bar{x_j}\|_{\varphi,w} \leq 1 +\epsilon$. Therefore, by Lemma \ref{Werner}, we obtain the desired result.
\end{proof}

A closed linear subspace $Y$ of the dual space $X^*$ of a Banach space $X$ is said to be {\it norming} if $\|f\|_X = \sup\{|F(f)| : F \in Y,\, \|F\|_{X^*} \leq 1\}$. It is well-known that the K\"othe dual space $X'$ is isometrically isomorphic to a closed norming subspace of $X^*$ if  the Banach function lattice $X$ has the Fatou property \cite[Proposition 1.b.18]{LT2}. The relationship between the SD2P and M-ideals is described as below. 

\begin{Lemma}\label{lem:mideal}\cite[Theorem 4.10]{ALN}
	Let $Y$ be a proper M-ideal of $X$, i.e. $X^* = Y^{\perp}\oplus_1 Z$, where $\oplus_1$ denotes the $L_1$-direct sum of $Y^{\perp}$ and a subspace $Z$ of $X^*$. If $Z$ is a norming subspace of $X^*$, then both $X$ and $Y$ have the SD2P.   
\end{Lemma} 

In our case  $Z=(\Lambda_{\varphi,w})'$ or $Z=(\lambda_{\varphi,w})'$  is a norming subspace of $(\Lambda_{\varphi,w})^*$ or $(\lambda_{\varphi,w})^*$, respectively. Now we are ready to provide a sufficient condition for Orlicz-Lorentz spaces with the SD2P. It has been shown in \cite{KLT} that the same sufficient condition on N-functions provides the Orlicz-Lorentz spaces with the D2P. Our result not only covers a broader class of Orlicz-Lorentz spaces than the previous one, but also shows that not satisfying the appropriate $\Delta_2$-condition is still sufficient to identify the Orlicz-Lorentz spaces with the SD2P, which is stronger than the D2P.

\begin{Theorem}\label{th: OLSD2P}
	\begin{enumerate}[\rm(i)]
		\item Let $\varphi$ be a nondegenerate Orlicz function and let $w$ be a weight function. If $\varphi$ does not satisfy the appropriate $\Delta_2$-condition, then $(\Lambda_{\varphi,w})_a$ and $\Lambda_{\varphi,w}$ has the SD2P. Hence, the spaces $(\Lambda_{\varphi,w})_a$ and $\Lambda_{\varphi, w}$ have the D2P and the LD2P.
		\item Let $\varphi$ be a nondegenerate Orlicz function and let $w$ be a weight sequence. If $\varphi$ does not satisfy the $\Delta_2^0$-condition, then $(\lambda_{\varphi,w})_a$ and $\lambda_{\varphi,w}$ has the SD2P. Hence, the spaces $(\lambda_{\varphi,w})_a$ and $\lambda_{\varphi, w}$ have the D2P and the LD2P.
	\end{enumerate}
\end{Theorem}

\begin{proof}
	We only prove (i) because (ii) can be obtained by the same argument. If an Orlicz function does not satisfy the appropriate $\Delta_2$-condition, the order-continuous subspace $(\Lambda_{\varphi,w})_a$ is a proper M-ideal in $ \Lambda_{\varphi,w}$ by Theorem \ref{th: MidealOL}. Then, both $(\Lambda_{\varphi,w})_a$ and $\Lambda_{\varphi,w}$ have the SD2P in view of Lemma \ref{lem:mideal}.
\end{proof}

 We summarize what we have shown so far. First, we recall the definitions of various octahedralities that are known to be the dual concepts to the D2Ps.

\begin{Definition}\cite{G, HLP} Let $X$ be a Banach space. \label{def:oct}
	\begin{enumerate}[\rm(i)]
		\item The norm on $X$ is octahedral if, for every finite-dimensional subspace $E \subset X$ and every $\epsilon > 0$, there exists $y \in S_X$ such that
		\[
		\|x + y\| \geq (1- \epsilon)(\|x\| + \|y\|) \,\,\, \text{for all} \,\,\, x \in E.
		\]
		\item The norm on $X$ is weakly octahedral if, for every finite-dimensional subspace $E \subset X$, every $x^* \in X^*$ and every $\epsilon > 0$, there exists $y \in S_X$ such that
		\[
		\|x + y\| \geq (1- \epsilon)(x^*x + \|y\|) \,\,\, \text{for all} \,\,\, x \in E.
		\] 
		\item The norm on $X$ is locally octahedral if, for every $x \in X$ and every $\epsilon >0$, there exists $y \in S_X$ such that
		\[
		\|sx + y\| \geq (1- \epsilon)(|s|\|x\| + \|y\|) \,\,\, \text{for all} \,\,\, s \in \mathbb{R}.
		\]
	\end{enumerate}
\end{Definition}

 Based on what we have observed in this section, we can now characterize the D2Ps of a class of Orlicz-Lorentz spaces as well as the octahedralities of their K\"othe duals.
 
\begin{Theorem}\label{th:equivD2P}
	Let $\varphi$ be a nondegenerate N-function at infinity and $w$ be a weight function. Then the following statements are equivalent for Orlicz-Lorentz spaces $\Lambda_{\varphi, w}$.
	\begin{enumerate}[\rm(i)]
		\item The function $\varphi$ does not satisfy the appropriate $\Delta_2$-condition.
		\item $\Lambda_{\varphi, w}$ has the SD2P.
		\item $\Lambda_{\varphi, w}$ has the D2P.
		\item $\Lambda_{\varphi, w}$ has the LD2P.
		\item The order-continuous subspace $(\Lambda_{\varphi, w})_a$ has the SD2P.
		\item The order-continuous subspace $(\Lambda_{\varphi, w})_a$ has the D2P.
		\item The order-continuous subspace $(\Lambda_{\varphi, w})_a$ has the LD2P.
		\item $\mathcal{M}_{\varphi_*, w}^0$ is octahedral.
		\item $\mathcal{M}_{\varphi_*, w}^0$ is weakly octahedral.
		\item $\mathcal{M}_{\varphi_*, w}^0$ is locally octahedral.
	\end{enumerate}  
	The similar facts also hold for the Orlicz-Lorentz sequence space $\lambda_{\varphi, w}$, its order-continuous subspace $(\lambda_{\varphi,w})_a$, and the space $\mathfrak{m}_{\varphi_*,w}^0$, defined by an Orlicz function $\varphi$ and a weight sequence $w$.
\end{Theorem}

\begin{proof}
	Here we only prove the function space case because the sequence space case uses the same argument. The implications (ii) $\implies$ (iii) $\implies$ (iv) , (v) $\implies$ (vi) $\implies$ (vii), and (viii) $\implies$ (ix) $\implies$ (x) are well-known facts \cite{ALN, HLP}. We have (i) $\implies$ (ii) and (i) $\implies$ (v) from Theorem \ref{th: OLSD2P}. Now, suppose that $\varphi$ is an N-function at infinity and satisfies the appropriate $\Delta_2$-condition. Then the Orlicz-Lorentz space $\Lambda_{\varphi,w} = (\Lambda_{\varphi,w})_a$ has the Radon-Nikod\'ym property by Theorem \ref{th:RNPsuff}. Hence we see that (iv) $\implies$ (i) and (vii) $\implies$ (i). By the duality result between D2Ps and octahedralities, the space $(\Lambda_{\varphi,w})_a$ has the SD2P (resp. D2P, LD2P) if and only if $(\Lambda_{\varphi,w})_a^* \simeq (\Lambda_{\varphi,w})_a' = \mathcal{M}_{\varphi_*,w}^0$ is octahedral (resp. weakly octahedral, locally octahedral) \cite{HLP}. 
\end{proof}

As a special case, we can obtain the similar characterization for Orlicz spaces $L_{\varphi}([0, \gamma), m)$, where $\gamma \leq \infty$ and $m$ is the Lebesgue measure. However, it can be shown that the same argument also holds over nonatomic, $\sigma$-finite measure spaces \cite{KLT}. Thus, we provide the more general statement for Orlicz spaces instead.  

\begin{Corollary}\label{th:equivD2POrlicz}
		Let $\varphi$ be a nondegenerate N-function at infinity and $\mu$ be a nonatomic $\sigma$-finite measure. Then the following statements are equivalent for Orlicz spaces $L_\varphi = L_{\varphi}(\Omega, \Sigma, \mu)$:
		\begin{enumerate}[\rm(i)]
			\item The function $\varphi$ does not satisfy the appropriate $\Delta_2$-condition.
			\item $L_\varphi$ has the SD2P.
			\item $L_\varphi$ has the D2P.
			\item $L_\varphi$ has the LD2P.
			\item The order-continuous subspace $(L_\varphi)_a$ has the SD2P.
			\item The order-continuous subspace $(L_\varphi)_a$ has the D2P.
			\item The order-continuous subspace $(L_\varphi)_a$ has the LD2P.
			\item $L_{\varphi_*}^0$ is octahedral.
			\item $L_{\varphi_*}^0$ is weakly octahedral.
			\item $L_{\varphi_*}^0$ is locally octahedral.
		\end{enumerate}  
		The analogous statements  also hold for  Orlicz sequence space $\ell_\varphi$ defined by an Orlicz function $\varphi$. 
\end{Corollary}

\section{The Daugavet property and the diametral D2Ps of $\Lambda_{\varphi,w}$}

In this section, we study the diametral D2Ps and the Daugavet property of Orlicz-Lorentz function space generated by a finite Orlicz function and equipped with the Luxemburg norm. A rearrangement invariant Banach function space defined on a finite interval with the Daugavet property must be isometric to $L_1$ \cite[Corollary 4.9]{KMMW}. However, it is unknown whether the same statement holds for an infinite measure space case.  The problem has been solved for Orlicz spaces. As a matter of fact, it has been shown that the space $L_1$ is the only Orlicz function space equipped with the Luxemburg norm that has the Daugavet property \cite[Theorem 4.12]{KLT}. We will use the relationship between $\Delta$-points and locally uniformly nonsquare points to identify Orlicz-Lorentz function spaces without the diametral D2Ps or the Daugavet property. 
\subsection{Locally uniformly nonsquare points are not $\Delta$-points}
A point $x\in S_X$ is said to be a {\it locally uniformly nonsquare point} (or {\it uniformly non-$\ell_1^2$ point}) if there exists $\delta > 0$ such that $\min\{\|x + y\|, \|x - y\|\} \leq 2 - \delta$ for all $y \in S_X$. Here the locally uniformly nonsquare points refer to those in the sense of James \cite{J, WSL}. A Banach space $X$ is {\it locally uniformly nonsquare (LUNSQ)} if every points on the unit sphere is a locally uniformly nonsquare point. It is well-known that uniformly nonsquare and locally uniformly rotund Banach spaces are LUNSQ. 

On the other hand, a Banach space $X$ is locally octahedral if every point on the unit sphere $S_X$ is not a locally uniformly nonsquare point. As a matter of fact, a Banach space with the Daugavet property or the DLD2P is locally octahedral. This comes from the fact that a Banach space $X$ having the DLD2P is equivalent to its dual $X^*$ having the $w^*$-DLD2P \cite[Theorem 3.5]{ALNT} and the duality relationship between the $w^*$-LD2P and the local octahedrality \cite{HLP}. Here we will show an alternative approach in terms of $\Delta$-points. The absence of locally uniformly nonsquare points has played an important role to identify Orlicz and Musielak-Orlicz spaces with the Daugavet property \cite{KK, KLT}, and we will use the similar approach for Orlicz-Lorentz spaces. First, we have the following lemma that will be useful.

\begin{Lemma}\label{lem:delta}\cite[Lemma 2.1]{JR}
	Let $X$ be a Banach space and $x\in S_X$ be a $\Delta$-point. For every $\epsilon > 0$, $\frac{\alpha}{1 - \alpha} < \epsilon$, and every slice $S = S(x^*, \alpha)$ containing $x$, there exists a slice $S(z^*, \alpha_1)$ of $B_X$ such that $S(z^*, \alpha_1) \subset S(x^*, \alpha)$ and $\|x - y\| > 2 - \epsilon$ for all $y \in S(z^*, \alpha_1)$. 
\end{Lemma} 

We now present one of the main ingredients in this section.
\begin{Theorem}\label{prop:nod}
	A locally uniformly nonsquare point $x \in S_X$ is not a $\Delta$-point of $X$. 
\end{Theorem}

\begin{proof}
		We show that a $\Delta$-point $x \in S_X$ cannot be a locally uniformly nonsquare point. Let $\epsilon > 0$ and $\eta \in (0, \frac{\epsilon}{2})$. By Lemma \ref{lem:delta}, for every $\alpha > 0$ where $\frac{\alpha}{1 - \alpha} < \eta$ and every slice $S = S(x^*, \alpha)$ containing $x$, there exists a slice $S(z^*, \alpha_1) \subset S$ such that $\|x - y\| > 2 - \eta > 2 - \epsilon$ for all $y \in S(z^*, \alpha_1)$. In particular, we have $\frac{z^*y}{\|y\|} > \frac{1-\alpha_1}{\|y\|} \geq 1 - \alpha_1$ for every $y \in S(z^*, \alpha_1)$. Hence $y' = \frac{y}{\|y\|} \in S(z^*, \alpha_1)$ and $\|x - y'\| > 2 - \epsilon$ for $y \in S(z^*, \alpha_1)$.   
		
		Moreover, by the fact that $\alpha < \frac{\alpha}{1 - \alpha} < \eta$ and $x, y \in S$, we have $\|x + y'\| \geq x^*x + x^*y' > 2 - 2\alpha > 2 - 2\eta > 2 - \epsilon$. Thus, for every $\epsilon > 0$ there exists $y' \in S_X$ such that $\min\{\|x + y'\|, \|x - y'\|\} > 2 - \epsilon$. This shows that $x \in S_X$ is not a locally uniformly nonsquare point.
\end{proof}

Recently, it has been shown that uniformly nonsquare Banach spaces do not contain $\Delta$-points \cite{ALMP}. Notice that we obtain a broader class of Banach spaces without $\Delta$-points by the previous theorem.

\begin{Corollary}\label{cor:nodelta}
	If a Banach space $X$ is LUNSQ, then $X$ does not have $\Delta$-points.
\end{Corollary}

We also have another approach to the well-known result.

\begin{Corollary}\cite[Theorem 4.5]{ALNT}\label{prop:LD2PLOH}
	If a Banach space $X$ has the DLD2P, then $X$ is locally octahedral.
\end{Corollary}

\begin{proof}
	If a Banach space $X$ has the DLD2P, then every point of the unit sphere $S_X$ is a $\Delta$-point. Then in view of Theorem \ref{prop:nod}, the space $X$ does not contain locally uniformly nonsquare points. Therefore, the space $X$ is locally octahedral. 
\end{proof}

\subsection{The Daugavet property and the diametral D2Ps in Orlicz-Lorentz spaces}
Now, we study the diametral D2Ps and the Daugavet property in Orlicz-Lorentz spaces using Theorem \ref{prop:nod}. First, we mention that the proof method in \cite{KLT} cannot be mimicked for Orlicz-Lorentz spaces. This hindrance comes from the fact the structure of Orlicz-Lorentz spaces also depends on various weight functions. 

The following lemma similar to \cite[Theorem 1.39.(4)]{Chen} plays an important role in this section.

\begin{Lemma}\label{lem:simple}
	Let $x \in (\Lambda_{\varphi,w})_a$. Then for every $\epsilon \in (0, 1)$, there exists $\delta \in (0, 1)$ such that $\rho_{\varphi,w}(x) \leq 1 - \epsilon$ implies $\|x\| \leq 1 - \delta$.
\end{Lemma}

\begin{proof}
	Assume to the contrary that for every $\epsilon \in (0, 1)$,  we have $\rho_{\varphi,w}(x) \leq 1 - \epsilon$ but $\|x\| = 1$. Consider an increasing sequence of simple functions $(x_n) \subset B_{(\Lambda_{\varphi,w})_a}$ such that $m(\supp{x_n}) < \infty$ and $|x_n|\uparrow |x|$ a.e. Then by \cite[Chapter 15, \S 72, Theorem 3]{Z} and by  theorem on the order-continuous subspace being a norm-closed order-ideal, we see that $\|x_n\|_{\varphi,w} \uparrow \|x\|_{\varphi,w} = 1$. Without loss of generality, we may assume that $\|x_n\|_{\varphi,w} \geq \frac{1}{2}$ for every $n \in \mathbb{N}$ so that $1 \geq \frac{1}{\|x_n\|_{\varphi,w}} - 1 = a_n \downarrow 0$ as $n \rightarrow \infty$.
	
	Now, let $K = \sup_n\{\rho_{\varphi,w}(2x_n)\}$. From the fact that $x \in (\Lambda_{\varphi,w})_a$, we have $\rho_{\varphi,w}(kx) <\infty$ for every $k>0$ \cite{K}. Hence by the monotonicity of $\rho_{\varphi,w}$, we have  $\rho_{\varphi,w}(2 x_n) \leq \rho_{\varphi,w}(2 x) < \infty$, which in turn implies that $K < \infty$. However, we also have
	\begin{eqnarray*}
		1 = \rho_{\varphi,w}\left(\frac{x_n}{\|x_n\|_{\varphi,w}}\right) = \rho_{\varphi,w}(2 a_n x_n + (1 - a_n)x_n) &\leq& a_n\rho_{\varphi,w}(2x_n) + (1 - a_n)\rho_{\varphi,w}(x_n)\\
		&\leq& K a_n + (1- \epsilon)\rightarrow 1 - \epsilon, 
	\end{eqnarray*}
	which leads to a contradiction. Therefore, there exists $\delta \in (0, 1)$ such that $\|x\|_{\varphi,w} \leq 1 - \delta$.
\end{proof}

The following fact on Orlicz functions will be useful.

\begin{Lemma}\cite[Lemma 4.1]{KLT}\label{lem:Orlicz}
Let $\varphi$ be an Orlicz function. For every closed and bounded interval $J \subset (d_{\varphi}, \infty)$ there exists a constant $\sigma \in (0, 1)$ such that $2\varphi(u/2)/\varphi(u) \leq \sigma$ for $u \in J$. 
\end{Lemma}

A Banach lattice $X$ is {\it locally uniformly monotone} if for every $x \in X_+$, with $\|x\| = 1$ and any $\epsilon \in (0, 1)$, there exists $\delta(x, \epsilon) \in (0,1)$ such that $0 \leq y \leq x$ and $\|y\| \geq \epsilon$ implies $\|x - y| \leq 1 - \delta(x, \epsilon)$. 
We mention that the Lorentz space $\Lambda_{1, w}(0, \gamma)$, $\gamma \leq \infty$, is always locally uniformly monotone from the assumptions on the weight function in this article \cite[Proposition 4.1]{FK}.

\begin{Theorem}\label{th:OLlunsq}
Let $w$ be a weight function and let $\varphi$ be an Orlicz function such that $d_\varphi < \infty$. Then there exists a locally uniformly nonsquare point in $(\Lambda_{\varphi,w})_a$.
\end{Theorem}

\begin{proof}
	Let $\varphi$ be an Orlicz function such that $d_\varphi < b_{\varphi} = \infty$. Choose $a > d_\varphi$ and  $A \in \Sigma$ such that $W(mA) = \frac{1}{\varphi(a)}$. Then there exists a closed bounded interval $J \subset (d_\varphi, \infty)$ including $a$.  Define a function $x = a\chi_A$. Clearly $\rho_{\varphi,w}(x) = \varphi(a)W(mA)  = 1$. Hence,  $\|x\|_{\varphi,w} = 1$ and $x \in (\Lambda_{\varphi,w})_a$.
	
	Now, we claim that $x $ is a locally uniformly nonsquare point in $(\Lambda_{\varphi,w})_a$. Let $y \in S_{(\Lambda_{\varphi,w})_a}$.  Pick $\eta \in (0, 1)$ and $t_1 \in [0, mA)$ such that $\varphi(a)W(t_1) \geq \eta$. Then there exists a measurable subset $E_{t_1} \subset A$ such that $\int_0^{t_1}x^*(t)dt  = \int_{E_{t_1}} x(t)dt$ \cite[Property 7$^\circ$, pg 64]{KPS}. Since $a \in J \subset (d_{\varphi}, \infty)$, by Lemma \ref{lem:Orlicz}, there exists $\sigma \in (0, 1)$ such that $\varphi\left(\frac{1}{2}a\right) \leq \frac{1 - \sigma}{2}\varphi(a)$.
	
	Consider
	\[
	B_1 = \{t \in E_{t_1}: x(t)y(t) \geq 0\}\,\,\, \text{and} \,\,\, B_2 = E_{t_1} \setminus B_1.	
	\]
	
	Since $\rho_{\varphi, w}(x \chi_{E_{t_1}}) = \varphi(a) W(t_1) \geq \eta$, we have either $\rho_{\varphi,w}(x \chi_{B_1}) \geq \frac{\eta}{2}$ or $\rho_{\varphi,w}(x \chi_{B_2}) \geq \frac{\eta}{2}$. Assume that $\rho_{\varphi,w}(x \chi_{B_1}) \geq \frac{\eta}{2}$. Then for almost all $t \in B_1$, 
	\begin{align*}
	\varphi\left(\frac{|x(t) - y(t)|}{2}\right) 
	 &\leq \varphi\left(\frac{\max\{x(t), y(t)\}}{2}\right) \leq \varphi\left(\frac{a}{2}\right) + \varphi\left(\frac {y(t)}{2}\right)\\
	 &\leq \frac{1}{2}\varphi(x(t)) - \frac{\sigma}{2}\varphi(|x(t)|) + \frac{1}{2}\varphi|(y(t)|).
	\end{align*}  
Therefore for a.a. $t\in I$,
	\begin{align*}
	\varphi\left(\frac{|x(t)-y(t)|}{2}\right)& = 
	\varphi\left(\frac{|x(t) - y(t)|}{2}\chi_{B_1}(t) + \frac{|x(t) - y(t)|}{2}\chi_{(B_1)^c}(t)\right)\\
	&\leq \frac{1}{2}\varphi(x(t)) - \frac{\sigma}{2}\varphi(x(t))\chi_{B_1}(t) + \frac{1}{2}\varphi(|y(t)|),
	\end{align*}
and computing the norm in $\Lambda_{1,w}$ we get
	\begin{eqnarray*}
	\rho_{\varphi,w}\left(\frac{x-y}{2}\right) = \left\|\varphi\left(\frac{|x - y|}{2}\right)\right\|_{1,w} &\leq& \left\|\frac{1}{2}\varphi(x) - \frac{\sigma}{2}\varphi(x)\chi_{B_1} + \frac{1}{2}\varphi(|y|)\right\|_{1,w}\\
	&\leq& \frac{1}{2}\left\|\varphi(x) - \sigma\varphi(x)\chi_{B_1}\right\|_{1,w} + \frac{1}{2}\left\|\varphi(|y|)\right\|_{1,w}.
	\end{eqnarray*}
	
Now by the fact that $\|\sigma\varphi(x)\chi_{B_1}\|_{1,w} = \sigma \rho_{\varphi,w}(x\chi_{B_1}) \ge \frac{\eta\cdot \sigma}{2}$ and by the local uniform monotonicity of $\Lambda_{1,w}$ \cite[Proposition 4.1]{FK}, there exists a constant $\delta_1\left(\varphi(x), \frac{\eta\cdot \sigma}{2}\right) > 0$ such that 
\[
\rho_{\varphi,w}\left(\frac{x-y}{2}\right) \le  1 - \delta_1\left(\varphi(x), \frac{\eta\cdot \sigma}{2}\right).
\]

On the other hand, if $\rho_{\varphi,w}(x \chi_{B_2}) \geq \frac{\eta}{2}$,   using  the similar argument as above there exists $\delta_2\left(\varphi(x), \frac{\eta\cdot \sigma}{2}\right) > 0$ with
\[
\rho_{\varphi,w}\left(\frac{x + y}{2}\right) \leq 1 - \delta_2\left(\varphi(x), \frac{\eta\cdot \sigma}{2}\right).
\]
	Hence, for $x = a\chi_A \in S_{(\Lambda_{\varphi,w})_a}$, there exists $\delta = \min\{\delta_1\left(\varphi(x), \frac{\eta\cdot \sigma}{2}\right), \delta_2\left(\varphi(x), \frac{\eta\cdot \sigma}{2}\right)\}$ such that
	\[
	\min\left\{\rho_{\varphi,w}\left(\frac{x-y}{2}\right),\rho_{\varphi,w}\left(\frac{x+y}{2}\right)\right\} \leq 1 - \delta.
	\]
	In view of Lemma \ref{lem:simple},  there exists $\epsilon > 0$ such that $ \min\left\{\left\|\frac{x-y}{2}\right\|_{\varphi,w},\left\|\frac{x+y}{2}\right\|_{\varphi,w}\right\} \leq 1 - \epsilon$ for every $y \in S_{(\Lambda_{\varphi,w})_a}$. Therefore, $x = a \chi_A$ is a locally uniformly nonsquare point in $(\Lambda_{\varphi,w})_a$.
\end{proof}

Recall that a Banach space $X$ has the DLD2P if and only if every point on $S_X$ is a $\Delta$-point. Then we obtain the following consequence.

\begin{Theorem}\label{cor:OLnoDP}
	Let $\varphi$ be an Orlicz function and let $w$ be a weight function. If an Orlicz-Lorentz space $\Lambda_{\varphi,w}$ has the DLD2P,  or DD2P, or the Daugavet property,  then $\varphi$ is a linear function.
\end{Theorem}

\begin{proof}
	First, we show that $a_{\varphi} = 0$. Indeed, assume to the contrary that $a_{\varphi} > 0$. Then $d_{\varphi} = a_\varphi$. Since the DLD2P is inherited by M-ideals \cite{LP}, if $\Lambda_{\varphi,w}$ has the DLD2P, then $(\Lambda_{\varphi, w})_a$ also has the DLD2P. Hence, every point on $S_{(\Lambda_{\varphi,w})_a}$ is a $\Delta$-point. This  implies that no point on  $S_{(\Lambda_{\varphi,w})_a}$ can be a locally uniformly nonsquare point by Theorem \ref{prop:nod}. So in view of Theorem \ref{th:OLlunsq} we have $d_{\varphi} = a_\varphi = \infty$. However, this is a contradiction because the Orlicz function $\varphi$ is not identically zero.
	
	Now, assume that $0= a_{\varphi} \leq d_{\varphi}$. Again, if $\Lambda_{\varphi, w}$ has the DLD2P, then $(\Lambda_{\varphi, w})_a$ has the DLD2P, and so $S_{(\Lambda_{\varphi,w})_a}$ does not have locally uniformly nonsquare points by the same reasoning with Theorem \ref{prop:nod}. This implies that $d_\varphi = \infty$ by Theorem \ref{th:OLlunsq}. Hence, we must have $\varphi(u) = ku$ for some $k \geq 0$ and for every $u \geq 0$. Furthermore, notice that $k > 0$ because $a_\varphi = 0$. Therefore, the Orlicz function $\varphi$ must be a linear function.  
	
	The same statement holds for the Daugavet property and the DD2P because any Banach space with the either property has the DLD2P.
\end{proof}

Recall that  a Banach lattice $X$ is said to be {\it uniformly monotone} if for every $\epsilon  > 0$, there exists $\delta(\epsilon) > 0$ such that for all $0\leq y \leq x$ with $\|x\| = 1$ and $\|y\| \geq \epsilon$ implies $\|x -y\| \leq 1 - \delta(\epsilon)$. 

\begin{Theorem}\label{th:OLDauUN}
	Let $\varphi$ be an Orlicz function and let $w$ be a weight function. If the weight function $w$ is regular, then the Orlicz-Lorentz space $\Lambda_{\varphi,w}$ with the Daugavet property is isometrically isomorphic to $L_1$.  If $W(1) =1$ then $\Lambda_{\varphi,w}=L_1$ with equality of norms.
\end{Theorem}

\begin{proof}
	By Theorem \ref{cor:OLnoDP}, if the Orlicz-Lorentz space $\Lambda_{\varphi,w}$ has the Daugavet property, then $\varphi(u) = ku$ for some $k > 0$. Hence the space $\Lambda_{\varphi,w}$ coincides with $\Lambda_{1, w}$ as a set and $\|\cdot\|_{\varphi,w} = k\|\cdot\|_{1,w}$. Moreover, any Lorentz space $\Lambda_{1,w}$ is uniformly monotone if and only if the weight function  $w$ is regular \cite[Theorem 1]{HK}. So by \cite[Theorem 4.4]{AKM},  $\Lambda_{1, w}$ must be isometrically isomorphic to $L_1$.  If in particular $W(1) =1$ then by Theorem 4.1 in \cite{AKM}, $\Lambda_{\varphi,w}=L_1$ with equality of norms.
\end{proof}	
In the next result we state necessary conditions for $\Lambda_{w,\varphi}$ having the Daugavet property.
\begin{Theorem}
	Let $\varphi$ be an Orlicz function and let $w$ be a weight function. If an Orlicz-Lorentz space $\Lambda_{\varphi,w}$ has the DLD2P,  or DD2P, or the Daugavet property, then the Orlicz function $\varphi$ is linear and $\lim_{t \rightarrow 0+}\frac{t}{W(t)} > 0$. 
\end{Theorem}

\begin{proof}
	By Theorem \ref{cor:OLnoDP}, if the Orlicz-Lorentz space $\Lambda_{\varphi,w}$ has the Daugavet property, then the Orlicz function is linear. So the space $\Lambda_{\varphi,w}$ is isometrically isomorphic to $ \Lambda_{1,w}$. If $\lim_{t \rightarrow 0+}\frac{t}{W(t)} = 0$, then the Lorentz space $\Lambda_{1,w}$ is a separable dual space \cite[Proposition 4.1]{AKM2}, so the space has the RNP. Since every Banach space with the DLD2P does not have the RNP, we must have $\lim_{t \rightarrow 0+}\frac{t}{W(t)} > 0$. 
\end{proof}

We see that the RNP only rules out a certain class of $\Lambda_{1, w}$ without the Daugavet property and the diametral D2Ps. The answer to the following question will lead us to a characterization of Orlicz-Lorentz spaces with the Daugavet property.

\begin{Problem}
	Let $\gamma = \infty$ and assume that the weight function $w$ is not regular or $\lim_{t \rightarrow 0+}\frac{t}{W(t)} > 0$. Can the Lorentz space $\Lambda_{1,w}$ have the Daugavet property, DD2P, or DLD2P?
\end{Problem}

On the other hand, in view of Theorem \ref{prop:nod}, the existence of locally uniformly nonsquare points not only provides the characterization of the Daugavet property of Orlicz function and sequence spaces equipped with the Luxemburg norm \cite{KLT} but also the characterization of such spaces with the DD2P and DLD2P. Here we provide a more general statement. It is well-known that every Banach space with the Daugavet property has an isomorphic copy of $\ell_1$ \cite[Theorem 2.9]{KSSW}, but here we have something stronger for the order-continuous subspace $(L_{\varphi})_a$.
 
\begin{Corollary}
	Let $\varphi$ be an Orlicz function and $\mu$ be a nonatomic, $\sigma$-finite measure. Consider the following statements for $L_{\varphi} = L_{\varphi}(\Omega, \Sigma, \mu)$.
	\begin{enumerate}
		\item $L_{\varphi}$ has the Daugavet property.
		\item $L_{\varphi}$ has the DD2P.
		\item $L_{\varphi}$ has the DLD2P.
		\item $\varphi$ is linear on $[0,\infty)$ that is $L_{\varphi}$ is isometrically isomorphic to $L_1$.
		\item $(L_{\varphi})_a$ contains an order-isometric copy of $\ell_1$.
	\end{enumerate}
	Then 
	\begin{enumerate}[\rm(i)]
		\item When $\mu(\Omega) < \infty$, the statements (1)-(5) are equivalent.
		\item When $\mu(\Omega) = \infty$, we have (1) $\iff$ (2) $\iff$ (3) $\iff$ (4) $\implies$ (5). 
	\end{enumerate}
	On the other hand, no Orlicz sequence space $\ell_{\varphi}$ has the Daugavet property, DD2P, or DLD2P.   
\end{Corollary}

\begin{proof}
	The equivalence of $(1)\iff(4)$ has been shown in  \cite[Theorem 4.12]{KLT}. The implication $(1) \implies (2) \implies (3)$ is clear from their definitions. In view of Theorem \ref{prop:nod}, if $L_{\varphi}$ has the DLD2P, then the space does not have locally uniformly nonsquare point. Hence, we see that $d_{\varphi} = \infty$ by \cite[Theorem 4.11.(i)]{KLT}, in other words, the Orlicz function $\varphi$ is a linear function. 
	
	For (i), it is well-known that the space $(L_{\varphi})_a$ contains an order-isometric copy of $\ell_1$ if and only if $\varphi$ is linear on $[0,\infty)$ that is $L_{\varphi}$ is isometrically isomorphic to $L_1$ \cite[Corollary 3.3]{KM}. For (ii), since the space $(L_{\varphi})_a$ contains an order-isometric copy of $\ell_1$ if and only if $\varphi$ is linear around a neighborhood of zero \cite[Corollary 3.2]{KM}, we have $(4) \implies (5)$.
\end{proof}    
 
 We finish this article with an additional observation from Theorem \ref{th:OLlunsq}. In view of Theorem \ref{th:equivD2P}, the K\"othe dual space $\mathcal{M}_{\varphi,w}^0$ of Orlicz-Lorentz space is octahedral, weakly octahedral, or locally octahedral, if and only if a finite N-function at infinity $\varphi_*$ does not satisfy the appropriate $\Delta_2$-condition. Now, we show that this space never satisfies the LD2P.
 
\begin{Theorem}\label{th:KothenoLD2P}
	Let $\varphi$ be an N-function at infinity and let $w$ be a weight function. Then the space $\mathcal{M}_{\varphi, w}^0$ does not have the LD2P. As a consequence, the space $\mathcal{M}_{\varphi, w}^0$ does not have the D2P, the SD2P, the diametral D2Ps, and the Daugavet property. 
\end{Theorem}

\begin{proof}
		If $\varphi$ is an N-function at infinity, then its complementary function $\varphi_*$ is finite \cite[Lemma 3.4.(a)]{KLT}. Hence $b_{\varphi_*} = \infty$. Now, we claim that $d_{\varphi_*} < \infty$. Indeed, assume to the contrary that $d_{\varphi_*} =  \infty$. Then $\varphi_*(u) = ku$ for all $u\ge 0$, where $k > 0$ because $\varphi_*$ is not identically a zero function. From the fact that $ \varphi_{**} = \varphi$, this implies that $\varphi(u) = 0$ on the interval $[0, k]$ and $\varphi(u) = \infty$ for every $u > k$. But this leads to a contradiction because the N-function $\varphi$ is finite. 
		
		Hence in view of Theorem \ref{th:OLlunsq},   the space $(\Lambda_{\varphi_*,w})_a$ has a locally uniformly nonsquare point $x \in S_{(\Lambda_{\varphi_*,w})_a}$. Then for a given $\epsilon > 0$, there exists a $w^*$-slice $S(x, \epsilon)$ of $B_{((\Lambda_{\varphi_*,w})_a)^*}$ which diameter is less than two by \cite[Theorem 4.5.(a)]{KLT}. Notice that 
		\[
		((\Lambda_{\varphi_*,w})_a)^* \simeq (\Lambda_{\varphi_*,w})' \simeq \mathcal{M}_{\varphi,w}^0.
		\]
		Now, consider the canonical mapping $J: (\Lambda_{\varphi_*,w})_a \rightarrow ((\Lambda_{\varphi_*,w})_a)^{**}$, where $J(x)(x^*) = x^*x$ for $x^* \in ((\Lambda_{\varphi_*,w})_a)^* \simeq \mathcal{M}_{\varphi,w}^0$. For $\tilde{x} = J(x) \in  S_{(\mathcal{M}_{\varphi,w}^0)^*}$, the set $S(\tilde{x}, \epsilon) = \{y \in B_{\mathcal{M}_{\varphi,w}^0}: \tilde{x}(y) > 1 - \epsilon\}$ is a slice of the unit ball $B_{\mathcal{M}_{\varphi,w}^0}$. Furthermore, we see that
		\[
		S(\tilde{x}, \epsilon)  = \{x^* \in B_{((\Lambda_{\varphi_*,w})_a))^*}: J(x)x^* > 1 - \epsilon\} = S(x, \epsilon),
		\]
		and so $\text{diam}\,S(\tilde{x}, \epsilon) < 2$. Therefore, the space $\mathcal{M}_{\varphi,w}^0$ does not have the LD2P.
		
		Since the D2P, the SD2P, the diametral D2Ps, and the Daugavet property imply the LD2P, the second statement can be obtained immediately. 
\end{proof}

\end{document}